\documentclass{amsart}
\usepackage{amssymb}

\newtheorem{theorem}{Theorem}[section]
\newtheorem{lemma}[theorem]{Lemma}

\newtheorem{example}[theorem]{Example}

\newtheorem{proposition}[theorem]{Proposition}
\newtheorem{corollary}[theorem]{Corollary}
\newtheorem{remark}[theorem]{Remark}

\numberwithin{equation}{section}

\begin{document}

\title{A characterization of absolutely minimum attaining operators}

\author{J. Ganesh}
\address{J. Ganesh, Department of Mathematics, I.I.T. Hyderabad, Kandi, Sangareddy, Medak(Dist), Telangana 502285, India.}
%\curraddr{Department of Mathematics and Statistics,
%Case Western Reserve University, Cleveland, Ohio 43403}
\email{ma12p1003@iith.ac.in}
\thanks{The first author was supported in part by National Board for Higher Mathematics Grant No. 2/39(1)/2012/NBHM/R D-II/5209.}

\author{G. Ramesh}
\address{G. Ramesh, Department of Mathematics, I.I.T. Hyderabad, Kandi, Sangareddy, Medak(Dist), Telangana 502285, India.}
\email{rameshg@iith.ac.in}
\author{D. Sukumar}
\address{D. Sukumar, Department of Mathematics, I.I.T. Hyderabad, Kandi, Sangareddy, Medak(Dist), Telangana 502285, India.}
\email{suku@iith.ac.in}
%    General info


\date{\today}



\begin{abstract}
We study the spectral properties of positive absolutely minimum attaining operators defined on infinite dimensional complex Hilbert spaces and using that derive a characterization theorem for such type of operators. We construct several examples and discuss some of the properties of this class. Also, we extend this characterization theorem for general absolutely minimum attaining operators by means of the polar decomposition theorem.
\end{abstract}

\maketitle

%\section*{This is an unnumbered first-level section head}
%This is an example of an unnumbered first-level heading.

%% The correct journal style for \specialsection is all uppercase; a known bug
%% in amsart.cls prevents this, so input must be uppercase until it is fixed.
%\specialsection*{This is a Special Section Head}
%\specialsection*{THIS IS A SPECIAL SECTION HEAD}
%This is an example of a special section head%
%%%%%%%%%%%%%%%%%%%%%%%%%%%%%%%%%%%%%%%%%%%%%%%%%%%%%%%%%%%%%%%%%%%%%%%%
%\footnote{Here is an example of a footnote. Notice that this footnote
%text is running on so that it can stand as an example of how a footnote
%with separate paragraphs should be written.
%\par
%And here is the beginning of the second paragraph.}%
%%%%%%%%%%%%%%%%%%%%%%%%%%%%%%%%%%%%%%%%%%%%%%%%%%%%%%%%%%%%%%%%%%%%%%%%
%%%%%%%%%%%%%%%%%%%%%%%%%%%%%%%%%%%%%%%%%%%%%%%%%%%%%%%%%%%%%%%%%%%%%%%%%%%%%%%%%%%%%
\section{Introduction}
%%%%%%%%%%%%%%%%%%%%%%%%%%%%%%%%%%%%%%%%%%%%%%%%%%%%%%%%%%%%%%%%%%%%%%%%%%%%%%%%%%%%%
Throughout the article we shall be concerned with infinite dimensional complex Hilbert spaces, denoted by $H, H_1, H_2$ etc, which are not necessarily separable. A bounded linear operator $T\colon H\rightarrow H$ is said to be diagonalizable if there exists an orthonormal basis for $H$ consisting entirely of eigenvectors of $T$. The spectral theorem assures that every positive compact operator is diagonalizable. Recently, a more general class of diagonalizable  operators, namely positive absolutely norm attaining operators have been discussed in \cite{car1, ram, sat}. This class includes the set of all positive compact operators as its subclass. Analogously,  absolutely minimum attaining operators are defined. Let $T$ be a bounded operator from $H_1$ to $H_2$. The quantity, $$ m(T):= \inf\{ \left\| Tx \right \| : x\in S_{H_1}\}, $$  is called the \emph{minimum modulus} of $T$. The operator $T$ is said to be minimum attaining if there exists a $x_0\in H_1$ with $\left\|x_0\right\|=1$ such that $m(T)=\left\|Tx_0\right\|$. It is said to be absolutely minimum attaining if $T|_{M}$ is minimum attaining for any non zero closed subspace $M$ of $H_1$. The class of absolutely minimum attaining operators is first introduced by Carvajal and Neves in \cite{car2}. In \cite{gan}, the authors of the present manuscript have proved that positive absolutely minimum attaining operators defined on separable complex Hilbert spaces of infinite dimension, can have at most one eigenvalue with infinite multiplicity. Based on this fact, they have classified the absolutely minimum attaining operators into two types \cite[Definition 3.9]{gan}, the one whose \emph{modulus value} has no eigenvalue with infinite multiplicity, called as the \emph{first type} and the others called as the \emph{second type}, respectively. They have studied the structure and spectrum of absolutely minimum attaining operators of \emph{first type} and left the \emph{second type} case for the future work. The main purpose of this article is to drop the separability condition and prove a general characterization theorem that covers both the \emph{first and second types} and also to rectify a gap in \cite[Theorem 4.6]{gan}. We appropriately modify and use some of the techniques present  in \cite{sat}, to develop the theory for the class of absolutely minimum attaining operators.
%\begin{quotation}
%	Let $T\in \mathcal{B}(H)$ be positive. Then $T\in \mathcal{AM}(H)$ iff there exists a decomposition for $T$ of the form $T:= \alpha I-K+F$ where  $K$ is a positive compact operator with $\|K\|\leq \alpha$ and $F$ is a positive finite rank operator satisfying $KF=FK=0$.
%\end{quotation}
%The above theorem not only rectifies  for the \emph{first type}, the case when $K$ has finite dimensional null space. It also gives a characterization of the positive absolutely minimum attaining operators of \emph{second type}, the case when $K$ has infinite dimensional null space .

 The article is organized as follows: There are over all five sections. In the second section we fix all the basic notations and terminology that we are going to use later. In the third, we discuss some characterization theorems concerning the minimum attaining operators which we will frequently use in the forthcoming sections, fourth section is devoted for the study of the spectral properties of positive absolutely minimum attaining operators and in the final section we prove the main characterization theorem that we are looking for.

%%%%%%%%%%%%%%%%%%%%%%%%%%%%%%%%%%%%%%%%%%%%%%%%%%%%%%%%%%%%%%%%%%%%%%%%%%%%%%%%%%%%%
\section{Notations and Definitions}
%%%%%%%%%%%%%%%%%%%%%%%%%%%%%%%%%%%%%%%%%%%%%%%%%%%%%%%%%%%%%%%%%%%%%%%%%%%%%%%%%%%%%
As usual, the inner product and the induced norm are denoted by $\langle , \rangle $ and $\left\|. \right\|$ respectively. The unit sphere of a closed subspace $M$ of $H$ is denoted by $S_M$. If $N$ is a subset of $H$, then the closed linear span of $N$ is denoted by $\left[N\right]$.

The space of all bounded linear operators from $H_1$ to $H_2$ is denoted by $\mathcal B(H_1, H_2)$ and $\mathcal B(H, H)$ is denoted by $\mathcal B(H)$.  If $T\in \mathcal B(H_1, H_2)$, then the null space and the range space are denoted by $N(T)$ and $R(T)$, respectively. The adjoint of $T\in \mathcal B(H_1, H_2)$ is denoted by $T^*$. Let $T\in \mathcal B(H)$. Then $T$ is called normal if $TT^*=T^*T$, self-adjoint if $T^*=T$, positive if $T^*=T$ and $\langle Tx,x\rangle\geq0$ for all $x\in H$, we denote it by $T\geq 0$. An operator $T\in \mathcal{B}(H)$ is called \emph{compact} if $\overline{T(B)}$ is compact for every bounded subset $B$ of $H$.

The \emph{spectrum} is defined by, $\sigma (T)=\{\lambda \in \mathbb{C} \colon  T-\lambda I \hspace{0.1cm} \text{is not invertible in}\hspace{0.1cm} \mathcal B(H) \}$ and the \emph{point spectrum} is defined by, $\sigma_p (T)= \{ \lambda \in \mathbb{C}\colon T- \lambda I \hspace{0.1cm} \text{is not injective}\}$. A closed subspace $M$ of $H$ is said to be \emph{invariant} under $T\in\mathcal B(H)$ if $TM\subseteq M$ and \emph{reducing} if both $M$ and $M^{\perp}$ are invariant under $T$.
A closed subspace $M$ is reducing for $T$ if and only if $TM\subseteq M$ and $T^*M\subseteq M$.

The set of all minimum attaining operators from $H_1$ to $H_2$ is denoted by $\mathcal{M}(H_1,H_2)$ and $\mathcal{M}(H,H)$ by  $\mathcal{M}(H)$. Similarly, the set of all absolutely minimum attaining operators from $H_1$ to $H_2$ is denoted by $\mathcal{AM}(H_1,H_2)$ and $\mathcal{AM}(H,H)$ by $\mathcal{AM}(H)$. The set of all positive absolutely minimum attaining operators on $H$ is denoted by $\mathcal{AM^+}(H)$ that is, $\mathcal{AM^+}(H):=\{T\in \mathcal{AM}(H): T\geq 0\}$.

Let $T\in \mathcal{B}(H)$ and $T\geq 0$. Then there exists a unique  operator $S\in \mathcal{B}(H)$ such that $S\geq 0$ and $T=S^2$. The operator $T^{\frac{1}{2}}:=S$ is called the \emph{ square root} of $T$. For $T\in \mathcal{B}(H_1, H_2)$, we have $T^*T\in \mathcal{B}(H_1)$ and $T^*T\geq 0$. The operator $|T|:={(T^*T)}^{\frac{1}{2}}$ is called the \emph{modulus} of $T$ (for more details, see\cite{kub, ree}).

The polar decomposition theorem, stated below will be used in the sequel.

\begin{theorem}\cite[Theorem 5.89, page 406]{kub}
If $T\in \mathcal{B}(H_1, H_2)$, then there exists a partial isometry $V\in \mathcal{B}(H_1, H_2)$ such that $T=V|T|$ and $N(V)=N(|T|)$. Moreover, this decomposition is unique. That is, if $W\in \mathcal{B}(H_1, H_2)$ is a partial isometry  and $Q\in \mathcal{B}(H_1)$ is a positive operator such that  $T=WQ$ and $N(W)=N(Q)$, then $W=V$ and $Q=|T|$.
\end{theorem}

%\begin{definition}[\cite {ram}]
%	Let $T \in \mathcal{B}(H_1, H_2)$. Then $$ m(T):= inf\{ \left\| Tx \right \| : x\in S_{H_1}\} $$ is called the \emph{minimum modulus} of $T$.	
%\end{definition}
%\begin{definition}
%	Let $T \in \mathcal{B}(H_1, H_2)$. Then
%	\begin{enumerate}	
%		\item $T $ is said to be \emph{minimum attaining} if there exists an element $x_0 \in S_{H_1}$ such that $$m(T)= \left \| Tx_{0} \right \|.$$
%		\item An operator $T\in\mathcal{B}(H_1, H_2)$ is said to be \emph{absolutely minimum attaining} if for every non zero closed subspace $M$ of $H$, $T|_M$ is minimum attaining.
%	\end{enumerate}
% \end{definition}
Following version of the spectral theorem will be used frequently in the forthcoming sections.

\begin{theorem}(\cite{hel, goh})
	Let $K$ be a compact self-adjoint operator on $H$. Then there exists an orthonormal sequence  $\{e_n\}_{n=1}^{k}\left(k<\infty \, \text{or}\, k=\infty \right)$ of eigenvectors of $K$ and corresponding sequence of eigenvalues $\{\lambda_n\}_{n=1}^{k}$ such that the following are true;
	\begin{enumerate}
		\item
		%\begin{equation*}
		$Kx= \displaystyle \sum_{n=1}^{k} \lambda_n\langle x, e_n\rangle e_n,\text{for all}\, \, x\in H$
		%\end{equation*}
		\item $|\lambda_1|\geq |\lambda_2| \geq |\lambda_3|\geq \dots $
		%\item We have $k<\infty$ if and only if $\text{dim}R(K)<\infty$.
		\item If $k=\infty$, then $\lim\limits_{n\rightarrow \infty}\lambda_n=0$.
	\end{enumerate}	
\end{theorem}
We will also use the following version of the spectral mapping theorem in the  forthcoming sections.
    \begin{theorem}(\cite[page 42]{hal})
    	Let $T\in\mathcal{B}(H)$ be an operator and `$p$' be a polynomial with complex coefficients. Then, $$\sigma \left(p(T)\right)=p\left(\sigma(T)\right).$$
	
	\end{theorem}
%%%%%%%%%%%%%%%%%%%%%%%%%%%%%%%%%%%%%%%%%%%%%%%%%%%%%%%%%%%%%%%%%%%%%%%%%%%%%%%%%%%%%
\section{Preliminary results}
%%%%%%%%%%%%%%%%%%%%%%%%%%%%%%%%%%%%%%%%%%%%%%%%%%%%%%%%%%%%%%%%%%%%%%%%%%%%%%%%%%%%%
In this section we discuss some of the basic properties of minimum attaining operators that we will be used later on. Recall that  $T\in \mathcal{B}(H_1, H_2)$ is said to be minimum attaining if there exists a $x_0\in S_{H_1}$ such that $$\left\|Tx_0\right\|=m(T)= \inf\{ \left \| Tx \right \| : x\in S_{H_1}\}.$$
\begin{proposition}
	\label{prop:eigenvalue}
	Let $T\in \mathcal{B}(H)$ be self-adjoint. Then $T\in \mathcal{M}(H)$ iff either $m(T)$ or $-m(T)$ is an eigenvalue of $T$. In particular, when $T\geq 0$, we have $T\in \mathcal{M}(H)$ iff $m(T)$ is an eigenvalue of $T$.
\end{proposition}
\begin{proof}
	The above result is proved  in \cite[Proposition 3.1, page 459]{gan} for the case of separable infinite dimensional complex  Hilbert spaces. The same proof works without any changes for non separable complex Hilbert spaces as well.
\end{proof}
\begin{proposition}
	\label{prop1}
	Let $T\in \mathcal{B}(H_1, H_2)$. Then the following statements are equivalent:
	\begin{enumerate}
		\item $T\in \mathcal{M}(H_1, H_2)$\\
		\item $|T| \in \mathcal{M}(H_1)$\\
		\item $T^*T \in \mathcal{M}(H_1)$.\\
		%\item $T^*\in \mathcal{M}(H_2, H_1)$\\
		%\item $|T^*| \in \mathcal{M}(H_2)$\\
		%\item $TT^* \in \mathcal{M}(H_2)$
	\end{enumerate}
\end{proposition}

\begin{proof}
	$(1)\Leftrightarrow (2)$:\\
	We have, ${\left\|Tx\right\|}_{H_2}^2={\langle Tx, Tx \rangle}_{H_2} = {\langle |T|x, |T|x \rangle}_{H_1}={\left\||T|x\right\|}_{H_1}^2, \, \forall x\in H_1  $. So it follows that $T\in \mathcal{M}(H_1, H_2)$ iff $ |T| \in \mathcal{M}(H_1) $.\\
	
	Note that $T^*T\geq 0$. By \cite[Proposition 2.2]{car2}, we have the following;
	\begin{align*}
		m(T^*T)& = \inf \left\{\langle T^*Tx, x \rangle \colon x\in S_{H_1} \right\}\\
		& = \inf \left\{\langle {|T|}^2x, x \rangle \colon x\in S_{H_1} \right\}\\
		& = \inf \left\{{\left\||Tx|\right\|}^2 \colon x\in S_{H_1} \right\}\\
		& ={\left[m(|T|)\right]}^2.
	\end{align*}
	
	$(2)\Rightarrow (3)$: Let $|T| \in \mathcal{M}(H_1)$. We have  $|T|\geq 0$, so by Proposition \ref{prop:eigenvalue}, $\exists \, x_0 \in S_{H_1}$ such that $|T|x_0= m(|T|)x_0$. This implies, ${|T|}^2x_0= {\left[m(|T|)\right]}^2x_0$ and so $T^*Tx_0= m(T^*T)x_0$. Consequently, $T^*T \in \mathcal{M}(H_1)$.\\
	
	$(3)\Rightarrow (2)$: Let $T^*T \in \mathcal{M}(H_1)$. Then by Proposition \ref{prop:eigenvalue}, $\exists\, x_0  \in S_{H_1}$ such that
	\begin{equation}
	\label{eqn:3.1}
	T^*Tx_0= m(T^*T)x_0.
	\end{equation}
	We have the following two cases,
	
	Case(I): $m(T^*T)=0$
	
	By \eqref{eqn:3.1}, $T^*Tx_0=0$. Since $N(T^*T)=N(T)=N(|T|)$, we get $ |T|x_0=0$ and so $|T| \in \mathcal{M}(H_1)$.
	
	Case(II): $m(T^*T)={\left[m(|T|)\right]}^2>0$
	
	By \eqref{eqn:3.1}, we have ${|T|}^2x_0={\left[m(|T|)\right]}^2x_0$. Consequently, $$\left[ \left (|T|+m(|T|)I\right )\left (|T|-m(|T|)I\right ) \right] (x_0)=0.$$ Since $|T|+m(|T|)I$ is invertible, we get $|T|x_0=m(|T|)x_0$ and so $|T| \in \mathcal{M}(H_1)$.
\end{proof}	
	%The equivalence of (4), (5) and (6) can be proved similarly by replacing $T$ with $T^*$ and $H_1$ with $H_2$ in the above proofs.\\
	
	%$(1)\Rightarrow (3)$: Let $T\in \mathcal{M}(H_1, H_2)$. Then by definition $\exists \, x_0 \in S_{H_1}$ such that $\left\|Tx_0\right\|= m(T)$. This implies, $\langle Tx_0, Tx_0 \rangle= {m(T)}^2\langle x_0, x_0\rangle$ and $\langle (T^*T- {m(T)}^2I)x_0, x_0 \rangle = 0$. Now, $m(T^*T)={m(T)}^2$ implies $T^*T-{m(T)}^2I$ is positive and thus  $T^*Tx_0={m(T)}^2x_0$. Consequently, we have
	%\begin{equation*}
	%\left\| T^*\left(\frac{Tx_0}{m(T)}\right) \right\|=m(T)=m(T^*)
	%\end{equation*}
	%\begin{equation*}
	% T^*\left\(Tx_0
	%\end{equation*}

\begin{corollary}
	\label{cor:normal}
	Let $T\in \mathcal{B}(H)$ be normal. Then $T\in \mathcal{M}(H)$ if and only if $T^*\in \mathcal{M}(H)$.
\end{corollary}
\begin{proof}
	We have $T$ is normal and so $|T^*|=|T|$. Now the proof is immediate from Proposition \ref{prop1}.
\end{proof}
\begin{remark}
	Note that Corollary \ref{cor:normal} is not valid in general for all minimum attaining operators. We have the following example.
\end{remark}

\begin{example}
	Let $T:{\ell}^2 \rightarrow {\ell}^2$ be defined by, %$$T(x_1, x_2, x_3,  \dots, x_n, \dots)= (x_2, \frac{x_3}{2}, \frac{x_4}{3}, \dots, \frac{x_n}{n-1},\dots), \forall x=(x_1, x_2, x_3, \dots, x_n, \dots )\in {\ell}^2.$$
	$$ T(e_{n})=\begin{cases}
	0\ , & \text{if}\ n = 1 , \\
	\frac{1}{n-1}e_{n-1}\ , & \text{if}\ n\geq 2.
	\end{cases} $$
	Then $m(T)=0=\|Te_1\|$, where $e_1=( 1, 0, 0, 0,\dots) \in S_{{\ell}^2}$. Hence  $T\in \mathcal{M}({\ell}^2)$. Now, the adjoint of $T$ is the map $T^*\colon {\ell}^2 \rightarrow {\ell}^2$ which satisfies,  %$$T^*(x_1, x_2, x_3,\dots, x_n, \dots )= (0,x_1, \frac{x_2}{2}, \dots , \frac{x_n}{n},\dots), \forall x=(x_1, x_2, x_3, \dots, x_n, \dots )\in {\ell}^2.$$
	$$T^*(e_{n})= \frac{1}{n}e_{n+1} , \, \,  \forall \, n \in \mathbb{N} .$$
	Then we have, $$0\leq m(T^*) \leq \|T^*e_n\|=\frac{1}{n}, \, \, \forall \, n\in \mathbb{N}.$$ Consequently, $m(T^*)=0$. But $\|T^*x\|=0$ implies that $x=0$. Hence $T^*\notin \mathcal{M}({\ell}^2)$. Here we notice that $T$ is not normal.
	\end {example}
%%%%%%%%%%%%%%%%%%%%%%%%%%%%%%%%%%%%%%%%%%%%%%%%%%%%%%%%%%%%%%%%%%%%%%%%%%%%%%%%%%%%%
\section{ Spectral properties}
%%%%%%%%%%%%%%%%%%%%%%%%%%%%%%%%%%%%%%%%%%%%%%%%%%%%%%%%%%%%%%%%%%%%%%%%%%%%%%%%%%%%%
In this section we investigate for some of the important properties satisfied by the spectrum of positive absolutely minimum attaining operators. Recall that $T\in \mathcal{B}(H_1, H_2)$ is said to be absolutely minimum attaining if $T|_M$ is minimum attaining for any non zero closed subspace $M$ of $H_1$.
\begin{theorem}
	\label{thm:diagonal}
	Let $T\in \mathcal{AM^+}(H)$. Then $T$ is diagonalizable.
\end{theorem}

\begin{proof}
	The proof follows in the similar lines to that of \cite[Theorem 3.8]{sat}. Let $\mathcal{B}$ be the collection of all orthonormal sets of eigenvectors of $T$. Since $T\geq 0$ and $T\in \mathcal{M}(H)$, by Proposition \ref{prop:eigenvalue}, we have $m(T)$ is an eigenvalue of $T$ and there exists a corresponding eigenvector for $T$, so $\mathcal{B}\neq \emptyset$. The elements of  $\mathcal{B}$ can be ordered by inclusion, and every chain $\mathcal{C}$ in  $\mathcal{B}$ has an upper bound, given by the union of all elements of $\mathcal{C}$. Thus, Zorn’s Lemma \cite[ page 267]{goh} assures the existence of a maximal element $B$ in  $\mathcal{B}$. Let $B=\{u_{\lambda} \colon \lambda \in \Lambda \}$ be the maximal orthonormal set of eigenvectors of $T$. We claim that $B$ is an orthonormal basis for $H$. \\
	
	Let $H_0 := \left[B\right]$. We claim that $H_0 = H$. It is enough to prove that ${H_0}^{\perp} =\{0\}$. Firstly, we observe that ${H_0}^{\perp}$ is invariant under $T$. Let $\mathcal{F}:=\left\{ F\subseteq \Lambda \colon F\, \text{is finite} \right\}$. Then given $x\in H_0$, we have
	\begin{equation*}
		x=\sum_{\lambda \in \Lambda}\langle x, u_\lambda \rangle u_\lambda =\lim_{F\in \mathcal{F}} \sum_{\lambda \in F}\langle x, u_\lambda \rangle u_\lambda .
	\end{equation*}
	Since the above net converges in the norm topology and $T$ is bounded, it follows that
	\begin{align*}
		Tx & =T\left( \lim_{F\in \mathcal{F}}\sum_{\lambda \in F}\langle x, u_\lambda \rangle u_\lambda \right)=\lim_{F\in \mathcal{F}}T\left( \sum_{\lambda \in F}\langle x, u_\lambda \rangle u_\lambda \right)\\
		& =\lim_{F\in \mathcal{F}} \sum_{\lambda \in F}\langle x, u_\lambda \rangle Tu_\lambda = \sum_{\lambda \in  \Lambda}\langle x, u_\lambda \rangle \alpha_{\lambda} u_\lambda \in H_0,
	\end{align*}
	considering $Tu_{\lambda}= \alpha_{\lambda} u_{\lambda}$, where $\alpha_{\lambda} \geq 0$ for every $\lambda \in \Lambda$. This shows that $H_0$ is an invariant subspace for $T$. Since $T^*=T$, it follows that $H_{0}^{\perp}$ is also invariant under $T$.\\
	
	It remains to show that ${H_0}^{\perp} =\{0\}$. Suppose, on the contrary that ${H_0}^{\perp}\neq\{0\}$. Then ${H_0}^{\perp}$ is a nontrivial closed subspace of $H$. Now $T$ is a positive absolutely minimum attaining operator implies that  $T|_{{H_0}^{\perp}}$ is also positive and attains its minimum on ${H_0}^{\perp}$.
	Consequently, $\|T|_{{H_0}^{\perp}}\|$ is an eigenvalue of $T|_{{H_0}^{\perp}}$. Let $u$ be a unit eigenvector of $T|_{{H_0}^{\perp}}$ corresponding to the eigenvalue $\|T|_{{H_0}^{\perp}}\|$. Clearly, $u\in{H_0}^{\perp}$ and so $u \in {B}^{\perp}$, a contradiction to the maximality of $B$ and we conclude that ${H_0}^{\perp} = \{0\}$. This shows that $H$ has an orthonormal basis consisting  of eigenvectors of $T$. Hence $T$ is diagonalizable.
\end{proof}

\begin{corollary}
	\label{cor1:diagonal}
	Let $T\in \mathcal{AM^+}(H)$. Then we have,
	\begin{equation*}
		Tx=\displaystyle \sum_{\lambda \in \Lambda}\alpha_{\lambda} \langle x, u_{\lambda}\rangle u_{\lambda}, \forall x\in H,
	\end{equation*}
	where $\{u_{\lambda}\colon \lambda \in \Lambda\}$ is an orthonormal basis for $H$ such that $Tu_{\lambda}=\alpha_{\lambda}u_{\lambda}, \forall \lambda \in \Lambda$. Moreover, for every nonempty subset $\Gamma$ of $\Lambda$, we have $\inf \{\alpha_{\lambda}\colon \ \lambda \in \Gamma\}=\min\{\alpha_{\lambda}\colon \ \lambda \in \Gamma\}$.
\end{corollary}

\begin{proof}
	We have that $T$ is diagonalizable by Theorem \ref{thm:diagonal}. Hence there exists an orthonormal basis $\{u_{\lambda}\colon \lambda \in \Lambda\}$  for $H$ such that $Tu_{\lambda}=\alpha_{\lambda}u_{\lambda}, \forall \lambda \in \Lambda$ with $\alpha_{\lambda}\geq 0$. Consequently, we have
	\begin{equation*}
		Tx=\displaystyle \sum_{\lambda \in \Lambda}\alpha_{\lambda} \langle x, u_{\lambda}\rangle u_{\lambda}, \forall x\in H.
	\end{equation*}
	It remains to prove the next part. Let us assume on the contrary that, $$\inf \{\alpha_{\lambda}\colon  \lambda \in \Gamma_0\} \neq \text{min} \{\alpha_{\lambda}\colon  \lambda \in \Gamma_0\},$$ for some nonempty subset $\Gamma_0$ of $\Lambda$. Let us denote by $\alpha=\inf \{\alpha_{\lambda}\colon  \lambda \in \Gamma_0\}$ and $H_0:=\left[u_{\lambda} \colon \lambda \in \Gamma_0 \right]$. Then for every $x\in S_{H_0}$, we have 	
	\begin{align*}
		{\|T|_{H_0}x\|}^2 & ={\left\|\sum_{\lambda \in \Gamma_0}\alpha_{\lambda} \langle x, u_{\lambda}\rangle u_{\lambda}\right\|}^2= \sum_{\lambda \in \Gamma_0}{|\alpha_{\lambda}|}^2 {|\langle x, u_{\lambda}\rangle u_{\lambda}|}^2 \\
		& > \sum_{\lambda \in \Gamma_0}{|\alpha|}^2 {|\langle x, u_{\lambda}\rangle u_{\lambda}|}^2={|\alpha|}^2\sum_{\lambda \in \Gamma_0} {|\langle x, u_{\alpha}\rangle u_{\alpha}|}^2={|\alpha|}^2{\|x\|}^2\\
		& = {| \alpha |}^2={\left(\inf \{\alpha_{\lambda}\colon \ \lambda \in \Gamma_0\}\right)}^2\\
		& = {\left[m(T|_{H_0})\right]}^2.
	\end{align*}	
	This implies that $\|T|_{H_0}x\|>m(T|_{H_0})$ for every $x\in S_{H_0}$, which means that $T$ is not minimum attaining on $H_0$, a contradiction to $T\in \mathcal{AM}(H)$. This proves the assertion.
\end{proof}

    Using the techniques from \cite {sat}, we prove the following lemma.

\begin{lemma}
	\label{lemma:equal}
	Let $T\in \mathcal{AM^+}(H)$. Suppose there exists two increasing sequences of eigenvalues $\{a_{n}\}_{n\in \mathbb{N}}$ , $\{b_{n}\}_{n\in \mathbb{N}}$ of $T$ with corresponding orthonormal sequences of eigenvectors $\{f_{n}\}_{n\in \mathbb{N}}$ , $\{g_{n}\}_{n\in \mathbb{N}}$ such that $a_n\rightarrow a$ and $b_n\rightarrow b$. Then we must have $a=b$.
\end{lemma}

\begin{proof}
	 Suppose $a\neq b$, then we have either $a < b $ or $b<a$. Let us consider the case $a<b$ and the other case can be dealt similarly. Without loss of generality we may assume that  $a<b_1$ so that $a_n <b_n$ for each $n\in \mathbb{N}$ (otherwise, we can choose a natural number $m$ such that $a<b_m$, redefine the sequence $(b_{n})_{n=m}^{\infty}$ by $(\widetilde{b}_{n})_{n=1}^{\infty}$ and proceed). Note that $Tf_n=a_n f_n$ and $Tg_n=b_n g_n$ for each $n\in \mathbb{N}$. Define $H_0 :=\left[ t_n f_n + \sqrt{1-t_n^2} g_n\colon n \in \mathbb{N}\right]$, where $t_n\in [0, 1]$ are yet to be determined. Observe that $H_0$ is a closed subspace of $H$ and hence a Hilbert space by itself. Moreover, the set $e_n=t_n f_n + \sqrt{1-t_n^2} g_n$ serves to be an orthonormal basis of $H_0$. Now let us define a sequence $(c_{n})_{n\in \mathbb{N}}$ by $c_{n}=a+\frac{(b-a)}{2n}$ for each $n\in \mathbb{N}$. Then $(c_{n})_{n\in \mathbb{N}}$ is a strictly decreasing sequence such that for every $n\in \mathbb{N}, {a_n}^2 < {c_n}^2 < {b_n}^2$ and $\lim_{n \rightarrow \infty} c_n=a$. Notice that ${t_n}^2{a_n}^2+(1-{t_n}^2){b_n}^2$ is a convex combination of ${a_n}^2$ and ${b_n}^2$, and hence it follows that ${t_n}^2{a_n}^2+(1-{t_n}^2){b_n}^2 \in [{a_n}^2, {b_n}^2]$ for each $n\in \mathbb{N}$. In fact, by choosing the right value of ${t_n}^2 \in [0, 1]$,  ${t_n}^2{a_n}^2+(1-{t_n}^2){b_n}^2$ can give any point in the interval $[{a_n}^2, {b_n}^2]$. Let us then choose a sequence $(t_n)_{n\in \mathbb{N}}$ such that ${t_n}^2{a_n}^2+(1-{t_n}^2){b_n}^2={c_n}^2$. Now $H_0=[t_n f_n+\sqrt{(1-{t_n}^2)}g_n\colon n\in \mathbb{N}]$ gives that,	
	\begin{align*}
		{[m(T|_{H_0})]}^2 & =\text{inf}\{ {\|Tx\|}^2\colon x\in H_0, \|x\|=1 \}\\
		& \leq \text{inf}\{ {\|Te_n\|}^2\colon n\in \mathbb{N} \}\\
		& \leq \text{inf}\{ {t_n}^2 {a_n}^2 + (1-t_n^2){b_n}^2\colon n\in \mathbb{N} \}\\
		& \leq \text{inf}\{ {c_n}^2\colon n\in \mathbb{N} \}\\
		& \leq {a}^2.
	\end{align*}
	However, any $x\in H_0$ with $\|x\|=1$, can be written as,
	\begin{equation*}
	x=\displaystyle \sum_{n=1}^{\infty} s_{n}\left(t_n f_n+\sqrt{(1-{t_n}^2)}\right)g_n ,  \text{where}  \displaystyle \sum_{n=1}^{\infty} {|s_{n}|}^2=1.
\end{equation*}
	 Consequently,
	\begin{align*}
		{\|T|_{H_0}x\|}^2& = {\left\|T\left(\displaystyle \sum_{n=1}^{\infty}s_{n} (t_n f_n+\sqrt{(1-{t_n}^2}g_n)\right)\right \|}^2\\
		& = \displaystyle \sum_{n=1}^{\infty}{|s_{n}|}^2 \left({t_n}^2{a_n}^2+(1-{t_n}^2){b_n}^2\right)\\
		& > \sum_{n=1}^{\infty}{|s_{n}|}^2{c_n}^2 \\
		& > {a}^2.
	\end{align*}	
	This implies that for every element $x\in H_0$ with $\|x\|=1$ we have $\|T|_{H_0}x\|> a \geq m(T|_{H_0})$. Which means that $T$ is not minimum attaining on $H_0$, a contradiction to $T\in \mathcal{AM}(H)$. Hence our assumption $a<b$ is wrong. Similarly, by changing the roles of $a$ and $b$, we prove that $b<a$ cannot be true. So we must have $a=b$. This completes the proof.
\end{proof}

\begin{proposition}
	\label{prop:onelimitpoint}
	Let $T\in \mathcal{AM^+}(H)$. Then  $\sigma (T)$ has at most one limit point. Moreover, this unique limit point (if it exists) can only be the limit of an increasing sequence in the spectrum.
\end{proposition}

\begin{proof}
	Since $T \in \mathcal{AM^+}(H)$, by Corollary \ref{cor1:diagonal}, we have,
	\begin{equation*}
		Tx=\displaystyle \sum_{\lambda \in \Lambda}\alpha_{\lambda} \langle x, u_{\lambda}\rangle u_{\lambda}, \forall x\in H,
	\end{equation*}
	where $\{u_{\lambda}\colon \lambda \in \Lambda\}$ is an orthonormal basis for $H$ such that $Tu_{\lambda}=\alpha_{\lambda}u_{\lambda}, \forall \lambda \in \Lambda$. Hence the spectrum $\sigma (T)$ of $T$, is the closure of $\{\alpha_{\lambda}\}_{\lambda \in \Lambda}$, (see \cite[Problem 63, page 34]{hal}).
	
	Let `$a$' be any limit point of $\sigma (T)$. We prove that there exists an increasing sequence $(a_n)_{n\in \mathbb{N}}\subseteq \{ \alpha_{\lambda}\colon \lambda \in \Lambda \}\text{ such that}\, \, a_n \rightarrow a$. It is enough to prove that there are at most only finitely many terms of the sequence $(a_n)_{n\in \mathbb{N}}$ that are strictly greater than  $a$. Suppose not, for the moment, let us assume that there are infinitely many such terms. This implies, there exists a decreasing subsequence $(a_{n_k})$ such that $a_{n_k} \rightarrow a$ and for each $n_k \in \mathbb{N},\, a_{n_k} > a$. Let $H_0:= \left[ u_{n_k}\right]$, where  $\{u_{n_k}\}\subseteq \{u_{\lambda}\colon \lambda \in \Lambda\} $. Then $H_0$ is a closed subspace of $H$ and hence a Hilbert space by itself. We have $T|_{H_0}$ is positive and by \cite[Proposition 2.1]{ram} we get $m(T|_{H_0})=\text{inf} \{|a_{n_k}|\}= a$. However, for every $x= \displaystyle \sum_{n_k}s_{n_k}u_{n_k} \in H_0 $ with ${\|x\|}^2=\displaystyle \sum_{k}{|s_{n_k}|}^2=1$, we have
	\begin{equation*}
		{\|T|_{H_0}x\|}^2= {\|\displaystyle \sum_{n_k}s_{n_k}a_{n_k} u_{n_k} \|}^2= \displaystyle \sum_{n_k}{|s_{n_k}|}^2{|a_{n_k}|}^2  > {a}^2 \sum_{n_k}{|s_{n_k}|}^2= {a}^2.
	\end{equation*}
	This implies that $\|T|_{H_0}x\|>m(T|_{H_0})=a$ for every $x\in S_{H_0}$, which means that $T$ is not minimum attaining on $H_0$, a contradiction to $T\in \mathcal{AM}(H)$. This proves our claim.
	
	Next, let $a$ and $b$ be any two limit points (if exist) of the spectrum $\sigma (T)$ of $T$. By the discussion in the above paragraph, there exist two increasing sequences $(a_{n})_{n\in \mathbb{N}}\subseteq \{\alpha_{\lambda}\}_{\lambda \in \Lambda}$, $(b_{n})_{n\in \mathbb{N}}\subseteq \{\alpha_{\lambda}\}_{\lambda \in \Lambda} $ with corresponding orthonormal sequences of eigenvectors $\{f_{n}\}_{n\in \mathbb{N}}\subseteq \{u_{\lambda}\}_{\lambda \in \Lambda}  $ , $\{g_{n}\}_{n\in \mathbb{N}}\subseteq \{u_{\lambda}\}_{\lambda \in \Lambda}$ such that $a_n\rightarrow a$ and $b_n\rightarrow b$. Then by applying Lemma \ref{lemma:equal}, we get $a=b$. This shows that the limit point (if it exists) of the spectrum $\sigma (T)$ of $T$ is unique.
	\end{proof}

\begin{corollary}
	Let $T\in \mathcal{AM^+}(H)$. Then  $\sigma_p(T)$ is always a countable set.
\end{corollary}

\begin{proof}
	Suppose our claim is not true. Then the spectrum $\sigma(T)$ of $T$ will be an uncountable subset of $\mathbb{R}$ . So by the fact that {\it ``every uncountable subset of real numbers must have at least two limit points''}, $\sigma(T)$ will have two limit points, which is a contradiction to Proposition \ref{prop:onelimitpoint}. Therefore $\sigma_p(T)$ must be always a countable set.
\end{proof}

\begin{corollary}
	Let $T\in \mathcal{AM^+}(H)$. Then $T$ can have at most one eigenvalue with infinite multiplicity.
\end{corollary}

\begin{proof}
	Let $a$ and $b$ be two eigenvalues(if exist) of $T$ with infinite multiplicity. Let $\{f_n\}_{n\in \mathbb{N}}$ and $\{g_n\}_{n\in\mathbb{N}}$ be two infinite sequences of orthonormal eigenvectors corresponding to the eigenvalues $a$ and $b$. Consider the two sequences $\{a_n\}_{n\in \mathbb{N}}$ and $\{b_n\}_{n\in \mathbb{N}}$  such that $a_n=a,\,\forall n\in \mathbb{N}$ and $b_n=b,\,\forall n\in \mathbb{N}$. Obviously, the sequences $\{a_n\}_{n\in \mathbb{N}}$ and $\{b_n\}_{n\in \mathbb{N}}$ are increasing to $a$ and $b$ respectively. Then by applying Lemma \ref{lemma:equal}, we must have $a=b$. Therefore $T$ can have at most one eigenvalue with infinite multiplicity.
\end{proof}
	%\textbf{Alternative proof}:
	
	%Suppose our claim is not true. Let $\lambda, \mu \in \sigma_p(T), \lambda \ne \mu$ be such that  dim $ N(T-\lambda I)= \infty$ and dim $ N(T-\mu I)= \infty$. Note that $\lambda, \mu\geq 0$. Since $T$ is positive $N(T-\lambda I)\perp N(T-\mu I)$. Let $H_0 = N(T-\lambda I)\stackrel{\perp}{\oplus} N(T-\mu I)$. Then $H_0$ is closed and  $T(H_0)\subseteq H_0$. Let $T_{1}=T|_{H_0}$. Then $T_{1} =\lambda P_{M}+ \mu P_{{M}^{\perp}}$, where  $M=N(T_{1}-\lambda I)$. Since $M\oplus {M}^{\perp}= H_0$, we get $T_1= \mu I+ (\lambda-\mu)P_{M}$. Now $T\in \mathcal{AM}(H)$ implies that $T_{1} \in \mathcal{AM}(H_0)$. Consequently, we get $ P_{M}\in \mathcal{AM}(H_0) $  and by \cite[Theorem 3.10]{car2},  either $M=N(T_{1}-\lambda I)$ or $M^\perp=N(T-\mu I)$ has to be finite dimensional, which is not the case. Hence our claim must be true.

\begin{corollary}
	Let $T\in \mathcal{AM^+}(H)$. If $\sigma(T)=\overline{\{\alpha_{\lambda}\}}_{\lambda \in \Lambda}$ has a limit point $\alpha$ and an eigenvalue with infinite multiplicity $\hat{\alpha}$ then, $\alpha=\hat{\alpha}$.
\end{corollary}

\begin{proof}
	Since $\alpha$ is a limit point of $\sigma(T)$, by Proposition \ref{prop:onelimitpoint}, there exists an increasing sequence $\{a_{n}\}_{n\in \mathbb{N}}\subseteq \{\alpha_{\lambda}\}_{\lambda \in \Lambda} $  such that $a_n \rightarrow \alpha $. Let $\{b_{n}\}_{n\in \mathbb{N}}\subseteq \{\alpha_{\lambda}\}_{\lambda \in \Lambda} $ be the constant sequence such that $b_n=\hat{\alpha},\, \forall n\in \mathbb{N}$. Let us denote by $\{f_n\}_{n\in \mathbb{N}}$ and $\{g_n\}_{n\in \mathbb{N}}$ the orthonormal sequence of eigenvectors corresponding to the eigenvalues $\{a_n\}_{n\in \mathbb{N}}$ and $\{b_n\}_{n\in \mathbb{N}}$ respectively. Clearly,  $b_n$ is increasing to $\hat{\alpha}$. Now, by applying Lemma \ref{lemma:equal}, we get $\alpha=\hat{\alpha}$.
\end{proof}
All the results discussed above concerning the spectrum of a positive absolutely minimum attaining operator put together will lead to the following theorem.
\begin{theorem}(compare with \cite[Theorem 3.8]{sat})
	\label{thm:spectrum}
	Let $T\in \mathcal{AM^+}(H)$. Then we have,
	\begin{equation*}
		Tx=\displaystyle \sum_{\lambda \in \Lambda}\alpha_{\lambda} \langle x, u_{\lambda}\rangle u_{\lambda}, \forall x\in H,
	\end{equation*}
	where $\{u_{\lambda}\colon \lambda \in \Lambda\}$ is an orthonormal basis for $H$ such that $Tu_{\lambda}=\alpha_{\lambda}u_{\lambda}, \forall \lambda \in \Lambda$ and the following hold true:
	\begin{enumerate}
		\item for every nonempty subset $\Gamma$ of $\Lambda$, we have $\inf \{\alpha_{\lambda}\colon\lambda \in \Gamma\}=\min\{\alpha_{\lambda}\colon\lambda \in\Gamma\}$;
		\item the spectrum $\sigma (T)=\overline{\{\alpha_{\lambda} \}}_{\lambda \in \Lambda}$ of $T$ has at most one limit point. Moreover, this unique limit point (if it exists) can only be the limit of an increasing sequence in the spectrum;
		\item the point spectrum $\sigma_p(T)$ of $T$ is countable and there can exist at most one eigenvalue for $T$ with infinite multiplicity;
		\item if the spectrum $\sigma (T)$ of $T$ has both, a limit point $\alpha$ and an eigenvalue $\hat{\alpha}$ with infinite multiplicity, then $\alpha=\hat{\alpha}$.
	\end{enumerate}
\end{theorem}

%%%%%%%%%%%%%%%%%%%%%%%%%%%%%%%%%%%%%%%%%%%%%%%%%%%%%%%%%%%%%%%%%%%%%%%%%%%%%%%%%%%%%
\section{Characterization}
%%%%%%%%%%%%%%%%%%%%%%%%%%%%%%%%%%%%%%%%%%%%%%%%%%%%%%%%%%%%%%%%%%%%%%%%%%%%%%%%%%%%%
In this section we prove a characterization theorem for positive absolutely minimum attaining operators that is similar to \cite[Theorem 5.1]{sat}. First we discuss some sufficient conditions to be satisfied by this class of operators.
\begin{lemma}
	\label{lemma:finiterank}
	Let $F\in \mathcal{B}(H)$ be a self-adjoint finite rank operator. Then for every $\alpha \geq 0$, we have $\alpha I - F\in \mathcal{M}(H)$.
\end{lemma}

\begin{proof}
	Let the range of $F$ be $k$-dimensional. Since $F$ is self-adjoint, by the spectral theorem there exists an orthonormal basis $B=\{ u_{\lambda}\colon \lambda \in \Lambda \} $ for $H$ corresponding to which the matrix of $F$ is  diagonal with $k$ non zero real diagonal entries, say $\{\alpha_1, \alpha_2,\dots,\alpha_k\}$. This implies that the matrix of $T:=\alpha I-F$ with respect to $B$ is also diagonal and consequently, $\sigma (T)=\{\alpha-\alpha_1,\alpha-\alpha_2, \alpha-\alpha_3, \dots, \alpha-\alpha_k, \alpha\}$. Note that  $T$ is self-adjoint. Now by using \cite[Proposition 2.1]{ram} we get,
	\begin{align*}
		m(T)& =d\left(0,\sigma(T)\right)\\
		& =\inf\{|\alpha-\alpha_1|,|\alpha-\alpha_2|, |\alpha-\alpha_3|,\dots, |\alpha-\alpha_k|, \alpha\}\\
		& = \min\{|\alpha-\alpha_1|,|\alpha-\alpha_2|, |\alpha-\alpha_3|,\dots, |\alpha-\alpha_k|, \alpha\}.
	\end{align*}
	It immediately follows that $T$ attains its minimum at  $u_{\lambda_0}$ for some $\lambda_0\in \Lambda$.
\end{proof}

Let $M$ be any closed subspace of $H$ and $i_M\colon M\rightarrow H$ be the inclusion map from $M$ to $H$, which is defined as $i_Mx= x, \forall x\in M$. Then it is easy to observe that the adjoint of $i_M$ is the map $i_M^*\colon H\rightarrow M$, which is defined as,

$$i_{M}^*(x)=\begin{cases}
x\ , & \text{if}\ x\in  M, \\
0\ , & \text{if}\ x\in  M^\perp.
\end{cases}$$

\begin{proposition}
	\label{prop:equivalent}
	Let $T\in \mathcal{B}(H)$. Then the following are equivalent;
	\begin{enumerate}
		\item $T\in \mathcal{AM}(H)$
		\item  $Ti_M\in \mathcal{M}(M, H)$ for every nonzero closed subspace $M$ of $H$
		\item  $i_M^*(T^*T)i_M \in \mathcal{M}(M)$ for every nonzero closed subspace $M$ of $H$.
	\end{enumerate}
\end{proposition}

\begin{proof}
	$(1)\Leftrightarrow (2)\colon$
	The proof is direct from the definition of absolutely minimum attaining operator, if we observe that $T|_M=Ti_M$ for every closed subspace of $M$ of $H$.\\
	$(2)\Leftrightarrow (3)\colon$
	It is a direct consequence of Proposition \ref{prop1}
\end{proof}

\begin{theorem}
	\label{thm:finiterank}
	Let  $F\in \mathcal{B}(H)$ be a finite rank operator. Then for every $\alpha \geq 0$ we have $\alpha I-F \in \mathcal{AM}(H)$.
\end{theorem}

\begin{proof}
	Let $T:=\alpha I-F$. Then we have $T^*=\alpha I-F^*$ and $T^*T=\beta I-\widetilde{F}$ where $\beta= {\alpha}^2$ and $\widetilde{F}= \alpha (F+F^*)-F^*F$ is a self-adjoint finite rank operator. Using Proposition \ref{prop:equivalent},
	%\begin{align*}
	%T\in \mathcal{AM}(H)  & \Leftrightarrow  \text{for every nonzero closed subspace} \,\, M \,\,\text{of}\,\, H, TJ_M\in \mathcal{M}(M, H)\\
	%                   & \Leftrightarrow  \text{for every nonzero closed subspace}\,\, M\,\,\text{of}\,\, H,{(TJ_M)}^* TJ_M\in \mathcal{M}(M)\\
	%                   & \Leftrightarrow  \text{for every nonzero closed subspace}\,\, M \,\, \text{of}\,\, H, {J_M}^*(T^*T)J_M \in \mathcal{M}(M)\\
	%                   & \Leftrightarrow  \text{for every nonzero closed subspace}\,\, M\,\,\text{of}\,\, H, {J_M}^*(\beta I-\widetilde{F})J_M \in \mathcal{M}(M)
	%\end{align*}
	it suffices to show that for every closed subspace $M$ of $H$, $i_{M}^*(\beta I-\widetilde{F})i_M \in \mathcal{M}(M)$. But $i_{M}^*(\beta I-\widetilde{F})i_M$ is a operator from the Hilbert space $M$ to itself and $i_{M}^*(\beta I-\widetilde{F})i_M=\beta(i_{M}^* I i_M)- i_{M}^*\widetilde{F}i_M=\beta I_M-\widetilde{F}_{M}$, where $\beta \geq 0$, ${I}_{M}$ is the identity operator on $M$ and $i_{M}^*\widetilde{F}i_M=\widetilde{F}_{M}$ is a self-adjoint finite rank operator on $M$. Now, Lemma \ref{lemma:finiterank} implies that $\beta I_M-\widetilde{F}_{M}\in \mathcal{M}(M)$. Hence the theorem.
\end{proof}
\begin{remark}
	As a particular case of the above theorem, it follows that $\alpha I-F \in \mathcal{M}(H)$, where $\alpha \geq 0$ and $F$ is any finite rank operator not necessarily self-adjoint.
\end{remark}

We know that finite rank operators, unitary operators and isometries are absolutely minimum attaining and the modulus of these operators is either a positive finite rank operator or the identity operator. In the first case, $0$ is  the eigenvalue with infinite multiplicity  and in the second case, $1$ is the eigenvalue with infinite multiplicity. Let $T\in \mathcal{AM^+}(H)$ and $\lambda$ be the eigenvalue of $T$ with infinite multiplicity. In general it is not true that, always either $\lambda=m(T)$ or $\lambda=\left\|T\right\|$. We have the following example to illustrate this.
%\begin{example}
%Let $K:\ell^2\rightarrow\ell^2$ be defined by
%	$$K(e_{n})=\begin{cases}
%	0\ , & \text{if}\ \text{n is odd}, \\
%	\frac{1}{n}e_n\ , & \text{if}\ \text{n is even}.
%		\end{cases}$$
%	Consider the operator $T:=I-K$. We have $T\geq 0$ and $T\in \mathcal{AM}({\ell}^2)$ by  \cite[Lemma 3.11]{car2}. Notice that $1$ is	the eigenvalue for $T$ with infinite multiplicity and $\left\|T\right\|=1$.		
%	\end{example}

%	\begin{example}
%	Let $P:\ell^2\rightarrow\ell^2$ be defined by
%	$$P(e_{n})=\begin{cases}
%		e_n\ , & \text{if}\ n\leq 10, \\
%		0\ , & \text{if}\ n\geq 10.
%\end{cases}$$
%	Consider the operator $T:=I+P$. Then we have  $T\geq 0$  and  $ {\left\|Tx\right\|}^2={\left\|x-Px\right\|}^2+4{\left\|Px\right\|}^2={\left\|x\right\|}^2+3{\left\|Px\right\|}^2$. Now, $P\in \mathcal{AM}({\ell}^2)$ implies $T\in \mathcal{AM}({\ell}^2)$. Notice that $1$ is the eigenvalue for $T$ with infinite multiplicity and $m(T)=1$.
%	\end{example}
%	After observing the above  examples, it is natural to expect the following,
%\begin{question}

%\end{question}

%We have the following example which answers the above question in negative.

\begin{example}
	Let $F:\ell^2\rightarrow\ell^2$ be defined by
	$$F(e_{n})=\begin{cases}
	\,\,\,\,\, e_n\ , & \text{if}\ n=1, 3, 5. \\
	-e_n\ , & \text{if}\ n=2, 4.\\
	\,\,\,\,\, 0 \ , & \text{if}\ n\geq 6.
	\end{cases}$$
	Consider the operator $T:=I-F$. Then we have  $T\geq 0$  and
	$T\in \mathcal{AM}(H)$ by Theorem \ref{thm:finiterank}. In this case, $1$ is the eigenvalue for $T$ with infinite multiplicity, which is different from $m(T)=0$ and $\left\|T\right\|=2$.
\end{example}

\begin{lemma}
	\label{lemma:finiterankcompact}
	Let $K\in \mathcal{B}(H)$ be a positive compact operator and $F\in \mathcal{B}(H)$ be a self-adjoint finite rank operator. Then for every $\alpha > 0$, we have $\alpha I-K+F \in \mathcal{M}(H)$.
\end{lemma}

\begin{proof}
	Firstly, if $K$ is of finite rank then from Lemma \ref{lemma:finiterank}, $T:=\alpha I-K+F \in \mathcal{M}(H)$.
	
	Next, assume that $K$ is of infinite rank. By the spectral theorem, there exists an orthonormal system of eigenvectors $\{ u_n\}_{n\geq 1}$ and
	corresponding eigenvalues $\{ \alpha_n\}_{n\geq 1}$ such that for all $x\in H$,
	\begin{equation}
		\label{eqn:spectralthm}
		(K-F)x=\displaystyle \sum_{n=1}^{\infty}\alpha_n \langle x, u_n\rangle u_n.
	\end{equation}
	Moreover, $\alpha_n\in \mathbb{R}, \, \forall n\in \mathbb{N}$ and $\{|\alpha_n|\}_{n\geq 1}$ is decreasing to $0$.
	Therefore for each $x\in H$, we have
	\begin{equation}
	\label{eqn:5.2}
		\langle (K-F)x, x\rangle =\displaystyle \sum_{n=1}^{\infty}\alpha_n {|\langle x, u_n\rangle|}^2.
	\end{equation}
	We claim that, there exists a $n_1$ such that $\alpha_{n_1}>0$. Suppose not, then by \eqref{eqn:5.2}, we have $0\leq K\leq F$. But $F$ is of finite rank and $K$ is positive, so it follows that $K$ is also of finite rank. In fact, for every $x\in {R(F)}^{\perp}=N(F)$ we have $0\leq \langle Kx, x\rangle \leq \langle Fx, x\rangle =0$ and so $\langle Kx, x\rangle=0$. Next, $K\geq 0$ implies that $Kx=0, \forall x\in  {R(F)}^{\perp}=N(F)$. Therefore we have $N(F)\subseteq N(K)$ and consequently $R(K)\subseteq R(F)$, which is a contradiction because $R(K)$ is infinite dimensional. Hence our claim is true.
	From Equation\eqref{eqn:spectralthm}, we have $\sigma(K-F)= \{ \alpha_n\}_{n=1}^{\infty} \cup \{0\}$ and the spectral mapping theorem gives that $\sigma(T)= \{ \alpha-\alpha_n\}_{n=1}^{\infty} \cup \{\alpha\}$. Now, \cite[Proposition 2.1]{ram} implies that $m(T)=d\left(0,\sigma(T)\right)=\text{inf}\{ \alpha, |\alpha-\alpha_n| \}_{n=1}^{\infty}$. But we know that $\{|\alpha_n|\}_{n\geq 1}$ is decreasing and $\alpha_{n_1}\geq 0$. This implies that $\alpha - \alpha_{n_1} \leq \alpha -\alpha_{n}, \, \forall n\geq n_1$. Next, $|\alpha_n| \rightarrow 0$ implies that there exists a $n_2$ such that $|\alpha_n|\leq \alpha$, $\forall n\geq n_2$. Consequently, $\alpha-\alpha_n\geq 0$, $\forall n\geq n_2$. Let $n_3= \text{max} \{n_1, n_2\}$. Then we have $|\alpha-\alpha_n| \geq |\alpha -\alpha_{n_3}|, \forall n\geq n_3$ and so $m(T)=\text{min}\{\alpha, |\alpha-\alpha_n| \}_{n=1}^{n_3}$. Clearly, $T$ attains its minimum either at $u_k$ for some $k \in \{1, 2, 3,\dots, n_3\}$ or at a unit vector in $N(K-F)$.
\end{proof}

\begin{theorem}
	\label{thm:sufficient}
	Let $K\in \mathcal{B}(H)$ be a positive compact operator and $F\in \mathcal{B}(H)$ be a finite rank operator. Then for every $\alpha \geq \frac {\|K\|}{2}$ we have $\alpha I-K+F \in \mathcal{AM}(H)$.
\end{theorem}

\begin{proof}
	Let $T:=\alpha I-K+F$. We prove the theorem in two cases as below.
	
	Case(I): $\alpha=0$
	
	In this case, $\alpha \geq \frac {\|K\|}{2}$ implies that $K=0$ and so $T$ is a finite rank operator. Therefore $T\in \mathcal{AM}(H)$.
	
	Case(II): $\alpha>0$
	
	We have $T^*T=\beta I-\widetilde{K}+\widetilde{F}$ where $\beta= {\alpha}^2$, $\widetilde{K}= 2\alpha K-K^2$  is a compact operator which is positive because $\alpha \geq \frac {\|K\|}{2}$  and $\widetilde{F}= \alpha (F +  F^*)-(KF+F^*K)+F^*F$ is a self-adjoint finite rank operator. Using Proposition \ref{prop:equivalent},
	%\begin{align*}
	%T\in \mathcal{AM}(H)  & \Leftrightarrow  \text{for every nonzero closed subspace} \,\, M \,\,\text{of}\,\, H, TJ_M\in \mathcal{M}(M, H)\\
	%                  & \Leftrightarrow  \text{for every nonzero closed subspace}\,\, M\,\,\text{of}\,\, H,{(TJ_M)}^* TJ_M\in \mathcal{M}(M)\\
	%                 & \Leftrightarrow  \text{for every nonzero closed subspace}\,\, M \,\, \text{of}\,\, H, {J_M}^*(T^*T)J_M \in \mathcal{M}(M)\\
	%                 & \Leftrightarrow  \text{for every nonzero closed subspace}\,\, M\,\,\text{of}\,\, H, {J_M}^*(\beta I-\widetilde{F})J_M \in \mathcal{M}(M)
	%\end{align*}
	it suffices to show that for every closed subspace $M$ of $H$, ${i_M}^*(\beta I-\widetilde{K}+\widetilde{F})i_M \in \mathcal{M}(M)$. But ${i_M}^*(\beta I-\widetilde{K}+\widetilde{F})i_M$ is an operator from the Hilbert space $M$ to itself and $i_{M}^*(\beta I-\widetilde{K}+\widetilde{F})i_M =\beta(i_{M}^* I i_M)- i_{M}^*\widetilde{K}i_M+i_{M}^*\widetilde{F}i_M=\beta I_M-\widetilde{K}_{M}+\widetilde{F}_M$, where, $\beta > 0$, $I_{M}$ is the identity operator on $M$, $\widetilde{K}_{M}=i_{M}^* \widetilde{K}i_M$ is a positive compact operator on $M$ and $\widetilde{F}_M=i_{M}^*\widetilde{F}i_M$ is a self-adjoint finite rank operator on $M$. Now, Lemma \ref{lemma:finiterankcompact} implies that $\beta I_M-\widetilde{K}_{M}+\widetilde{F}_{M}\in \mathcal{M}(M)$. Hence the theorem.
\end{proof}

    So far we could establish a number of sufficient conditions to be satisfied by absolutely minimum attaining operators, what follows is a necessary condition.
%\section{Characterization}

\begin{theorem}
	\label{thm:necessary}
	Let $T\in \mathcal{AM^+}(H)$. Then there exists a positive scalar $\alpha$, a positive compact operator $K$ and a  positive finite rank operator $F$ such that the following is true:
	\begin{enumerate}
		\item $T= \alpha I-K+F$;
		\item $\|K\| \leq \alpha$ and $KF=FK=0$.
	\end{enumerate}
\end{theorem}

\begin{proof}
	We have $T\in \mathcal{AM^+}(H)$. Then by Theorem \ref{thm:spectrum}(3), there can exist at most one eigenvalue for $T$ with infinite multiplicity.
	We prove the theorem in the following two cases separately.\\
	
	Case(I): $T$ has no eigenvalue with infinite multiplicity.\\
	
	To prove this case, we follow the approach used in \cite {gan}.
	Let $H_{1}:=H$ and $T_{1}:=T$. Since $ T \in \mathcal{AM}(H)$ and $T \geq 0 $ we get $T_{1} \in \mathcal{M}(H_{1})$ and
	$ T_{1} \geq 0 $. Then by Proposition~\ref{prop:eigenvalue}, there exists a $u_1 \in S_{H_1} $ such that $T_1u_1=m(T_{1})u_1$. Let $ \alpha_{1}=m(T_{1})$. Then  $\alpha_{1} \geq 0$.
	
	Let $H_2:={\left[u_1 \right]}^{\perp}$. Note that $H_1\supseteq H_2$ and $H_2$ reduces $T$. Let $T_2:=T|_{H_2}$. Since $T\in \mathcal{AM}(H)$ and $T\geq 0$, we get $T_2 \in \mathcal{M}(H_2)$ and $T_2\geq 0$. Then by Proposition \ref{prop:eigenvalue}, there exists a $u_2 \in S_{H_2} $ such that $T_2u_2=m(T_{2})u_2$. Let $ \alpha_{2}=m(T_{2})$. Then  $\alpha_{2} \geq \alpha_{1} \geq 0$ and  $u_1 \perp u_2$.
	
	Let $H_3:={\left[ u_1, u_2\right]}^{\perp}$.  Note that  $ H_1\supseteq H_2\supseteq H_3$ and $H_3$ reduces $T$. Let $T_3:=T|_{H_3}$. Since $T\in \mathcal{AM}(H)$ and $T\geq 0$  we get $T_3 \in \mathcal{M}(H_3)$ and $T_3\geq 0$. Then by Proposition \ref{prop:eigenvalue}, there exists a $u_3 \in S_{H_3} $ such that  $T_3u_3=m(T_3) u_3$. Let $\alpha_{3}=m(T_3)$. Then $ \alpha_{3}\geq\alpha_{2}\geq \alpha_{1} \geq 0$ and $u_3\perp u_i, i=1, 2.$
	
	Proceeding this way after $n$ many steps we get a sequence of subspaces
	$\{H_{i}\}_{i=1}^{n}$ of $H$ such that $H_{1} \supseteq H_{2} \supseteq H_{3}\dots \supseteq H_{n}$ where $H_i={\left[u_1, u_2, u_3,\dots,u_{i-1}\right]}^{\perp}$ for all $1\leq i\leq n $  and also a sequence of scalars $\{\alpha_{i}\}_{i=1}^{n}$ such that $0 \leq \alpha_{1}\leq \alpha_{2}\leq \alpha_{3}\leq \dots \leq \alpha_{n}$, where $\alpha_i=m(T_i)$ for all $1\leq i\leq n$.
	
	Next, we claim that $H_{n} \neq \{0\}$ for all $n \in \mathbb{N}$. If not, then there exists a $n\in  \mathbb{N}$ such that $H_{n}=\{0\}$. By the projection theorem we have,
	\begin{equation*}
		\label{eqn_1}
		H=H_{n} \oplus {H_n}^\perp={H_n}^\perp= \left[u_1, u_2, u_3,\dots,u_{n-1}\right],
	\end{equation*}
	a contradiction to $H$ is infinite dimensional. Therefore $H_{n}\neq \{0\}$ for all $n \in \mathbb{N}$.  So there exists an infinite sequence of scalars $ \{\alpha_{n}\}_{n\in \mathbb{N}}$ such that $0\leq \alpha_n\leq \alpha_{n+1}\leq \|T\|$ for all $n\in \mathbb{N}$. % We have
	%$\alpha_{n}\neq \left\|T\right\|$ for all $n\in \mathbb{N}$. If $\alpha_{n_0}=\left\|T\right\|$ for some $n_0\in \mathbb{N}$, then we get $H_{n_0+k+1}=\{0\}$, where $k$ is the multiplicity of $\left\|T\right\|$, which is not true by our claim above.%Therefore we have $ 0 \leq \alpha_{1}\leq \alpha_{2}\leq \alpha_{3}\leq \dots <\left\|T\right\|$.
	By the monotonic convergence theorem $\alpha_n \rightarrow \alpha$ for some  $\alpha \leq \left\|T\right\|$.
	
	Let $M_1=\left[u_n \colon n\in \mathbb{N}\right]$. Denote by $M_2=M_1^{\perp}$. We must have dim $M_2<\infty$, if not then by applying the same procedure as above, we can find an increasing sequence of eigenvalues of $T$ that converges to a scalar which is greater than $\alpha$, but this is  a contradiction to Theorem \ref{thm:spectrum}(2), that is $\sigma(T)$ can have at most one limit point.
	
	Denote by, $K:=\alpha P_{M_1}-TP_{M_1}$. Then we have, $Kx:=\displaystyle \sum_{n=1}^{\infty}(\alpha-\alpha_n)\langle x, u_n\rangle, \forall x\in H$. Now the converse of spectral theorem \cite[Theorem 6.2, page 181]{goh} gives that $K$ is a positive compact operator. Clearly, $\|K\|\leq \alpha$ and $\overline{R(K)}=M_1$.
	
	Denote by, $F:= TP_{M_2}-\alpha P_{M_2}$. Note that $F$ is a finite rank operator and $R(F)\subseteq M_2$. Next, $M_2$ is a reducing subspace for $T$ implies that $TP_{M_2}=P_{M_2}T$ \cite[Proposition 3.7, page 39]{con} and so $F$ is self-adjoint. Now, $m\left(T|_{M_2}\right)\geq \alpha_n, \forall n \in \mathbb{N}$ implies that $m\left(T|_{M_2}\right)\geq \alpha$. Therefore $\sigma_p (F)=\{\lambda -\alpha \colon \lambda \in \sigma_p \left(T|_{M_2}\right)\} \subseteq [0, \infty)$. Consequently, $F$ is positive.
	
	Clearly, $H=M_1\oplus M_2$. Then $T=TP_{M_1}+TP_{M_2}= (\alpha P_{M_1}-K)+(F+\alpha P_{M_2})=\alpha I-K+F$. Obviously, $KF=FK=0$.\\
	
	Case(II): $T$ has exactly one eigenvalue with infinite multiplicity $\alpha$ (say).\\
	
	Let $M_1:=N(T-\alpha I)$. By the Projection theorem, we have $H=M_1\oplus M_1^{\perp}$.
	
	Suppose dim $M_1^{\perp}<\infty$. Since $M_1$ is a reducing subspace for $T$, $T|_{M_1^{\perp}}$ is a positive finite rank operator. By the Spectral theorem, there exists an orthonormal eigenbasis for $M_1^{\perp}$. We have $\alpha  \notin \sigma_p\left(T|_{M^{\perp}}\right)$ since $T\geq 0$. Let $M_2$ be the subspace of $M_1^{\perp}$ which is spanned by all eigenvectors corresponding to the eigenvalues of $T|_{M_1^{\perp}}$ that are less than $\alpha$ (if no eigenvalues is smaller than $\alpha$ then we take   $M_2=\{0\}$). Similarly, define $M_3$ to be the subspace of $M_1^{\perp}$ that is spanned by all eigenvectors corresponding to the eigenvalues of $T|_{M_1^{\perp}}$ that are greater than $\alpha$ (if no eigenvalues is greater than $\alpha$ then we take  $M_3=\{0\}$). Clearly, $M_1^{\perp}=M_2\oplus M_3$ and $H=M_1 \oplus M_2 \oplus M_3$. Then we have,
	\begin{align*}
	T & =TP_{M_1}+TP_{M_2}+TP_{M_3}\\
	& =\alpha P_{M_1}+TP_{M_2}+TP_{M_3}\\
	& =\alpha I-(\alpha P_{M_2}-TP_{M_2})+(TP_{M_3}-\alpha P_{M_3}) \left[\because I=P_{M_1}+P_{M_2}+P_{M_3}\right].
	\end{align*}
	This implies that $T= \alpha I-K+F$ where $K=\alpha P_{M_2}-TP_{M_2}$ and $F=TP_{M_3}-\alpha P_{M_3}$ are both positive finite rank operators such that $KF=FK=0$ and $\|K\|\leq \alpha$.
	
	In case dim $M_1^{\perp}=\infty$, let $H_1:=M_1^{\perp}$. Note that $H_1$ reduces $T$. Denote by $T_1=T|_{H_1}$. Since $T\in \mathcal{AM}(H)$ and $T\geq 0$, we get $T_1 \in \mathcal{AM}(H_1)$ and $T_1\geq 0$. Moreover, $T_1$ has no eigenvalue with infinite multiplicity. Then by applying same procedure to $T_1$, like in Case(I), we can get an infinite orthonormal sequence of eigenvectors $\{u_n\}_{n=1}^{\infty}$ and a corresponding sequence of eigenvalues $\{\alpha_n\}_{n=1}^{\infty}$ of $T$ such that $ 0 \leq \alpha_{1}\leq \alpha_{2}\leq \alpha_{3}\leq \dots \leq \left\|T_1\right\|\leq \left\|T\right\|$ and $\alpha_n \rightarrow \beta$. From Theorem \ref{thm:spectrum}(4), we must have $\alpha=\beta$. Let us denote by, $M_2=\left[u_n \colon n\in \mathbb{N}\right]$. Let $M_3$ be the orthogonal compliment of $M_2$ in $H_1$. Then we must have dim $M_3<\infty$, otherwise $\sigma(T)$ will have two distinct limit points, which is not possible by Theorem \ref{thm:spectrum}(2).
	
	Denote by, $K:=\alpha P_{M_2}-TP_{M_2}$. Then we have, $Kx:=\displaystyle \sum_{n=1}^{\infty}(\alpha-\alpha_n)\langle x, u_n\rangle, \forall x\in H$. Now the converse of spectral theorem \cite[Theorem 6.2, page 181]{goh} gives that $K$ is a positive compact operator. Clearly, $\|K\|\leq \alpha$ and $\overline{R(K)}=M_2$.
	
	Denote by, $F:=TP_{M_3}-\alpha P_{M_3}$. Note that $F$ is a finite rank operator and $R(F)\subseteq M_3$. Next, $M_3$ is a reducing subspace for $T$ implies that $TP_{M_3}=P_{M_3}T$ \cite[Proposition 3.7, page 39]{con} and so $F$ is self-adjoint. Now, $m\left(T|_{M_3}\right)\geq \alpha_n, \forall n \in \mathbb{N}$ implies that $m\left(T|_{M_3}\right)\geq \alpha$. Therefore $\sigma_p (F)=\{\lambda -\alpha \colon \lambda \in \sigma_p \left(T|_{M_3}\right)\} \subseteq [0, \infty)$. Consequently, $F$ is positive.
	
	Finally, we have, $T=TP_{M_1}+TP_{M_2}+TP_{M_3}=\alpha I-(\alpha P_{M_2}-TP_{M_2})+(TP_{M_3}-\alpha P_{M_3})$. It follows that $T=\alpha I-K+F$.
\end{proof}

\begin{theorem}
	\label{thm:characterization}
	Let $T\in \mathcal{B}(H)$. Then the following are equivalent:
	\begin{enumerate}
		\item $T\in \mathcal{AM^+}(H)$;
		\item There exists a decomposition for $T$ of the form $T:= \alpha I-K+F$ where  $K$ is a positive compact operator with $\|K\|\leq \alpha$ and $F$ is a positive finite rank operator satisfying $KF=FK=0$. Moreover, this decomposition is unique.
	\end{enumerate}
\end{theorem}

\begin{proof}
	
	(1) $\Rightarrow$ (2):
	We have $T\in \mathcal{AM^+}(H)$. Then from Theorem \ref{thm:necessary}, $T$ is of the form, $T=\alpha I-K+F$ where $K$ is a positive compact operator with $\|K\|\leq \alpha$ and $F$ is a positive finite rank operator satisfying $KF=FK=0$.
	
	It remains to prove the uniqueness part:
	
	Let, if possible $T$ has another decomposition of the form $T:=\hat{\alpha}I-\hat{K}+\hat{F}$ where $\hat{K}$ is a positive compact operator with $\|\hat{K}\|\leq \alpha$ and $\hat{F}$ is a positive finite rank operator satisfying $\hat{K}\hat{F}=\hat{F}\hat{K}=0$.
	
	By the spectral mapping theorem, we have $\sigma(T)=\alpha-\sigma(K-F)$. Since $K-F$ is a self-adjoint compact operator and dim $H=\infty$, by applying the spectral theorem we get that $\alpha$ is either the limit point of $\sigma(T)$ or the eigenvalue of $T$ with infinite multiplicity. By the similar arguments, we get $\hat{\alpha}$ is also, either the limit point of $\sigma(T)$ or the eigenvalue of $T$ with infinite multiplicity. Now, Theorem \ref{thm:spectrum} implies that $\alpha=\hat{\alpha}$.  Next, $\alpha I-K+F=\hat{\alpha}I-\hat{K}+\hat{F}$ implies that,
	\begin{equation}
		\label{eqn:1}
		K-F=\hat{K}-\hat{F}.
	\end{equation}
	We also have, $(K+F)^2=(\hat{K}+\hat{F})^2$ because $KF=FK=\hat{K}\hat{F}=\hat{F}\hat{K}=0$, but every positive operator has a unique positive square root\cite[Theorem VI.9, page 196]{ree}, so we must have,
	\begin{equation}
		\label{eqn:2}
		K+F=\hat{K}+\hat{F}.
	\end{equation}
	Now, combining the Equations \eqref{eqn:1} and \eqref{eqn:2}, we get $K=\hat{K}$ and $F=\hat{F}$.\\

	(2) $\Rightarrow$ (1):
	We have $T\in \mathcal{B}(H)$  and $T$ is of the form $T:=\alpha I-K+F$ where $K$ is a positive compact operator such that $\|K\|\leq \alpha $ and $F$ is a finite rank operator. Then, $T\geq 0$ because $\alpha I-K \geq 0$ and $F\geq 0$. From Theorem \ref{thm:sufficient}, it follows that $T\in\mathcal{AM^+}(H)$.
\end{proof}

Let $T\in \mathcal{AM^+}(H)$. Then, according to \cite[Definition 3.9]{gan}, $T$ is of
\begin{enumerate}
	\item \emph{first type} if it has no eigenvalue of infinite multiplicity,
	\item \emph{second type} if it has a unique eigenvalue of infinite multiplicity.
\end{enumerate}
The next theorem completely characterizes the positive absolutely minimum attaining operators of both the types.

\begin{theorem}
	\label{thm:types}
	Let $T\in \mathcal{B}(H)$. Then the following are equivalent:
	\begin{enumerate}
	\item
	 $T\in \mathcal{AM^+}(H)$
	
	 \item  There exists a unique decomposition for $T$ of the form $T:= \alpha I-K+F$ where  $K$ is a positive compact operator with $\|K\|\leq \alpha$ and $F$ is a positive finite rank operator satisfying $KF=FK=0$. %\item
	 Moreover,
	\begin{enumerate}
		\item  $T$ is of \emph{first type} whenever $N(K-F)$ is finite dimensional,
		\item  $T$ is of \emph{second type} whenever $N(K-F)$ is infinite dimensional.
	\end{enumerate}
    \end{enumerate}
\end{theorem}

\begin{proof}
First part of the proof follows directly from Theorem \ref{thm:characterization}. Next, by spectral mapping theorem, we have $\sigma(T)=\alpha-\sigma(K-F)$. Clearly, $T$ has an eigenvalue of infinite multiplicity if and only if $N(K-F)$ is infinite dimensional. Hence the theorem.
\end{proof}	
\begin{remark}
	If $F=0$, in Theorem \ref{thm:types} 2(a), then $T=\alpha I-K$. In this case,  $\alpha = \|T\|$. This is exactly the structure theorem obtained for \emph{first type} positive absolutely minimum attaining operators \cite [Theorem 4.6]{gan}. But this is not the case always, for instance, we have an example below.
\end{remark}

\begin{example}
	Let $D:{\ell}^2 \rightarrow {\ell}^2$ be defined by,
	$$ D(e_{n})=\begin{cases}
	0\ , & \text{if}\ n = 1 , \\
	\frac{1}{n}e_{n}\ , & \text{if}\ n\geq 2.
	\end{cases} $$
Let $P$ be the orthogonal projection onto $[e_1]$. Consider the operator, $T:=I-D+P$.
We have $T\in \mathcal{AM^+}({\ell}^2)$ by Theorem \ref{thm:sufficient}. Clearly, $\alpha =1, \|T\|=2$ and $\alpha < \|T\|$.	
\end{example}

The following example shows that the set $\mathcal{AM}(H)$ is not closed under addition.
\begin{example}
	Let $U\colon {\ell}^2 \rightarrow {\ell}^2$ be defined as,
	$$U(e_n)= \left(\sqrt{\frac{1}{n}}+ i \sqrt{1-\frac{1}{n}}\right)e_n, \forall \, n \in \mathbb{N}$$
	Then $U\in \mathcal{AM}({\ell}^2)$ because it is unitary. Obviously, the identity operator $I\in \mathcal{AM}({\ell}^2)$. But $I+U \notin \mathcal{AM}({\ell}^2)$. In fact, for each $x\in S_{{\ell}^2}$ we have,
	\begin{equation*}
		{\|(I+U)x\|}^2  =\displaystyle \sum_{n=1}^{\infty}\left [{\left(1+\sqrt{\frac{1}{n}}\right)}^2+ \left(1-\frac{1}{n}\right)\right]{|x_n|}^2
		= \displaystyle \sum_{n=1}^{\infty}\left(2+2\sqrt{\frac{1}{n}}\right){|x_n|}^2
		> 2.
	\end{equation*}
	On the other hand, we have $\inf\|(I+U)e_n\|=\sqrt{2}$. Therefore, $m(I+U)=\sqrt{2}$. But $\|(I+U)x\|> \sqrt{2}, \forall \, x\in S_{{\ell}^2}$. Hence $I+U \notin \mathcal{AM}({\ell}^2)$.
\end{example}
%\begin{proposition}
%Let $P\in \mathcal{B}(H)$ be an orthogonal projection. Then $P\in \mathcal{AM}(H)$ iff either $N(P)$ or $R(P)$ is finite dimensional.
%\end{proposition}
%\begin{proof}
%Since $P$ is an orthogonal projection we have $P\geq 0$. Suppose $P\in \mathcal{AM}(H)$, then from Theorem \ref{thm:characterization}, $P$ has the decomposition of the form $P=\alpha I-K+F$, where $K$ is a positive compact operator with $\|K\|\leq \alpha $ and
%$F$ is a positive finite rank operator.
%\end{proof}
\begin{corollary}
	The class, ${\mathcal{AM}}^+(H)$ is closed under addition. In fact, ${\mathcal{AM}}^+(H)$ is a cone in the real Banach space of self-adjoint operators.
\end{corollary}

\begin{proof}
	Let $T_1, T_2 \in {\mathcal{AM}}^+(H)$. Then $T_1+T_2$ is obviously positive. Next, by Theorem \ref{thm:characterization}, there exists positive scalars $\alpha_1, \alpha_2$, positive compact operators $K_1, K_2$ and positive finite rank operators $F_1, F_2$ such that $T_1=\alpha_1I -K_1+F_1$ and $T_2=\alpha_2I -K_2+F_2$ where $\|K_1\| \leq \alpha_1$ and $\|K_2\| \leq \alpha_2$ . Moreover, $K_1F_1=F_1K_1=0$ and $K_2F_2=F_2K_2=0$. Then $T_1+T_2=(\alpha_1 +\alpha_2)I-(K_1+K_2)+(F_1+F_2)$. Now, by Theorem \ref{thm:sufficient}, $T_1 + T_2 \in \mathcal {\mathcal{AM}}^+(H)$. Suppose $T$ and $-T$ both are in ${\mathcal{AM}}^+(H)$, then $T=0$. This shows that ${\mathcal{AM}}^+(H)$ is a proper cone in the real Banach space of self-adjoint operators.
\end{proof}

\begin{theorem}
	\label{thm:modulusAM}
	Let $T\in \mathcal{B}(H_1, H_2)$. Then the following statements are equivalent:
	\begin{enumerate}
		\item $T\in \mathcal{AM}(H_1, H_2)$\\
		\item $|T| \in \mathcal{AM^+}(H_1)$\\
		\item $T^*T \in \mathcal{AM^+}(H_1).$
		%\item $T^*\in \mathcal{M}(H_2, H_1)$\\
		%\item $|T^*| \in \mathcal{M}(H_2)$\\
		%\item $TT^* \in \mathcal{M}(H_2)$
	\end{enumerate}
\end{theorem}

\begin{proof}
	
	$(1)\Leftrightarrow (2)$:
	We have,
\begin{equation*}
{\left\|Tx\right\|}_{H_2}^2={\langle Tx, Tx \rangle}_{H_2} = {\langle |T|x, |T|x \rangle}_{H_1}={\left\||T|x\right\|}_{H_1}^2, \, \forall x\in H_1.
 \end{equation*}
 Clearly, $T\geq 0$. It follows that $T\in \mathcal{AM}(H_1, H_2)$ iff $ |T| \in \mathcal{AM^+}(H_1) $.\\
	
	$(2)\Rightarrow (3)$: Let $|T| \in \mathcal{AM^+}(H_1)$. Then by Theorem \ref{thm:characterization}, there exists a decomposition for $|T|$ of the form $|T|:= \alpha I-K+F$ where  $K$ is a positive compact operator with $\|K\|\leq \alpha$ and $F$ is a positive finite rank operator satisfying $KF=FK=0$. This implies, $T^*T=|T|^2=\beta I-\widetilde{K}+\widetilde{F}$, where $\beta= {\alpha}^2$, $\widetilde{K}= 2\alpha K-K^2$  is a compact operator which is positive because $\alpha \geq \|K\|$  and $\widetilde{F}= 2\alpha F +  F^2$ is a positive finite rank operator.
	Next,  by the Spectral radius formula for normal operators \cite[Theorem 1]{ber} $, \|\widetilde{K}\|= \text{sup} \{2\alpha \lambda -{\lambda}^2 \colon \lambda \in \sigma(K)\}\leq {\alpha}^2=\beta$. Since $R(\widetilde{K})\subseteq R(K)$ and $R(\widetilde{F})\subseteq R(F)$, we have $\widetilde{K}\widetilde{F}=\widetilde{F}\widetilde{K}=0$. Then by applying Theorem \ref{thm:characterization} once again, we conclude that $T^*T\in \mathcal{AM^+}(H)$.\\
	
	$(3)\Rightarrow (2)$: Let $T^*T \in \mathcal{AM^+}(H_1)$. Then by Theorem \ref{thm:characterization}, there exists a decomposition for $T^*T$ of the form $T^*T:= \alpha I-K+F$ where  $K$ is a positive compact operator with $\|K\|\leq \alpha$ and $F$ is a positive finite rank operator satisfying $KF=FK=0$.
	
	By the spectral theorem, there exists an finite or infinite orthonormal sequence of eigenvectors $\{u_n\}_{n\geq 1}$ corresponding to the eigenvalues $\{\alpha_n\}_{n\geq 1}$ of $K$ such that,
	\begin{equation*}
		Kx= \displaystyle \sum_{n\geq 1}\alpha_n \langle x, u_n\rangle u_n, \forall x \in H.
	\end{equation*}
	Moreover, $ \alpha_n\geq \alpha_{n+1}\geq 0, \forall n\geq 1$ and in case if, $\{u_n\}_{n\geq 1}$ is an infinite sequence then $\alpha_n\rightarrow0$. Let us denote by $M_1=\left[u_n \colon n\geq 1\right]$.
	Similarly, by applying the spectral theorem to $F$, we get a finite orthonormal sequence of eigenvectors $\{v_n\}_{n=1}^{k}$ corresponding to the eigenvalues $\{\beta_n\}_{n= 1}^{k}$ of $F$ such that,
	\begin{equation*}
		Fx= \displaystyle \sum_{n=1}^{k}\beta_n \langle x, v_n\rangle v_n, \forall x \in H.
	\end{equation*}
	Moreover, $\beta_k\geq \beta_{k-1}\geq \dots \beta_2\geq \beta_1\geq 0$. Let us denote by $M_2=\left[v_n \colon 1\leq n\leq k\right]$. We have, $M_1\perp M_2$.
	Let $M_3$ be the orthogonal compliment of $M_1\oplus M_2$ in $H$. Clearly, $H=M_1\oplus M_2 \oplus M_3 $. Let $\{w_{\lambda}\}_{\lambda \in \Lambda}$ be an orthonormal basis for $M_3$. Using the decomposition for $T^*T$, we have,
	\begin{equation*}
		T^*Tx=\alpha \displaystyle \sum_{\lambda \in \Lambda}\langle x, w_{\lambda}\rangle w_{\lambda} +\displaystyle \sum_{n\geq 1}(\alpha-\alpha_n) \langle x, u_n\rangle u_n +\displaystyle\sum_{n=1}^{k}(\alpha+\beta_n) \langle x, v_n\rangle v_n , \forall x \in H.
	\end{equation*}
	Let us consider the operator $S:H\rightarrow H$ defined as,
	\begin{equation*}
		Sx={\alpha}^{\frac{1}{2}} \displaystyle \sum_{\lambda \in \Lambda}\langle x, w_{\lambda}\rangle w_{\lambda} +\displaystyle \sum_{n\geq 1}{(\alpha-\alpha_n)}^{\frac{1}{2}} \langle x, u_n\rangle u_n +\displaystyle\sum_{n=1}^{k}{(\alpha+\beta_n)}^{\frac{1}{2}} \langle x, v_n\rangle v_n , \forall x \in H.
	\end{equation*}
	Then $S\geq 0$ because we have $\alpha_n\leq \alpha, \forall n\geq 1$. By the definition, $S$ is positive square root of $T^*T={|T|}^2$, but positive square root is unique. Therefore we must have $S=|T|$. Let us define $K_1:H\rightarrow H$ by,
	\begin{equation*}
		K_1x:= \displaystyle \sum_{n\geq 1}\left({(\alpha-\alpha_n)}^{\frac{1}{2}}-{\alpha}^{\frac{1}{2}}\right)\langle x, u_n\rangle u_n, \forall x \in H.
	\end{equation*}
	If the set,$\{u_n\}_{n\geq 1}$ is finite then $K_1$ is a positive finite rank operator.
	In case if the set,$\{u_n\}_{n\geq 1}$ is infinite, we have the sequence  $\{{(\alpha-\alpha_n)}^{\frac{1}{2}}-{\alpha}^{\frac{1}{2}}\}_{n\geq 1}$ is monotonically decreasing to $0$. Now, the converse of Spectral theorem \cite[Theorem 6.2, page 181]{goh} gives that $K_1$ is a positive compact operator. Clearly, $\|K_1\|\leq {\alpha}^{\frac{1}{2}}$ and $R(K_1)=M_1$. Let us define $F_1:H\rightarrow H$ by,
	\begin{equation*}
		Fx=\displaystyle \sum_{n=1}^{k}\left({(\alpha+\beta_n)}^{\frac{1}{2}}-{\alpha}^{\frac{1}{2}}\right) \langle x, v_n\rangle v_n, \forall x \in H.
	\end{equation*}
	Clearly, $F_1$ is a positive finite rank operator and $R(F_1)=M_2$. Moreover, we have $K_1F_1=F_1K_1=0$. Now, it is easy to observe that $|T|={\alpha}^{\frac{1}{2}}I-K_1+F_1$. Then by Theorem \ref{thm:characterization}, we have $|T|\in \mathcal{AM^+}(H)$.
\end{proof}
\begin{corollary}
	Let $T\in \mathcal{B}(H)$ be positive. Then $T\in \mathcal{AM^+}(H)$ if and only if $T^{\frac{1}{2}} \in \mathcal{AM^+}(H)$.
\end{corollary}
\begin{proof}
	Since $T\geq 0$, Theorem \ref{thm:modulusAM} implies that $T\in \mathcal{AM^+}(H)$ if and only if $T^2 \in \mathcal{AM^+}(H)$. Consequently, $T\in \mathcal{AM^+}(H)$ if and only if $T^{\frac{1}{2}} \in \mathcal{AM^+}(H)$.
\end{proof}	
	
\begin{corollary}
	\label{cor:normalAM}
	Let $T\in \mathcal{B}(H)$ be normal. Then $T\in \mathcal{AM}(H)$ if and only if $T^*\in \mathcal{AM}(H)$.
\end{corollary}

\begin{proof}
	We have $T$ is normal and so $|T^*|=|T|$. Now the proof is immediate from Theorem \ref{thm:modulusAM}.
\end{proof}
\begin{remark}
	Note that Corollary \ref{cor:normalAM} is not valid in general for all absolutely minimum attaining operators. We have the following example.
\end{remark}

\begin{example}
	Let $V:{\ell}^2 \rightarrow {\ell}^2$ be defined by,
	$$ V(e_{n})= e_{2n}, \forall n\in \mathbb{N}.$$
	Then, we have $V\in \mathcal{AM}({\ell}^2)$ because it is an isometry. Now, the adjoint
	$V^*:{\ell}^2 \rightarrow {\ell}^2$ is be  given by,
	$$V^*(e_{n})= \begin{cases}
	0\ , & \text{if}\ n \, \text{is odd} \\
	e_{\frac{n}{2}}\ , & \text{if}\ n \, \text{is even}.
	\end{cases}$$ 	
	Then $V^*$ is a partial isometry such that both $N(V^*)$ and $R(V^*)$ are infinite dimensional, hence by Proposition \ref{prop:partialisometry}(which is proved later), $V^*\notin\mathcal{AM}({\ell}^2)$. Notice that $V$ is not normal.
\end{example}

\begin{proposition}
	Let $T\in \mathcal{AM}(H)$. Then either $N(T)$ or $R(T)$ is finite dimensional.
\end{proposition}

\begin{proof}
	%	We have $N(|T|)=N(T)$ and $R(|T|)=R(T)$. Moreover, Theorem \ref{thm:modulusAM} gives that, $T\in \mathcal{AM}(H)$ iff $|T|\in \mathcal{AM}(H)$. It is enough to prove the result for $|T|$.
		From Theorem \ref{thm:modulusAM}, $T\in \mathcal{AM}(H)$ implies that $T^*T \in \mathcal{AM^+}(H)$. Then by Theorem \ref{thm:characterization}, $T^*T$ has a decomposition of the form, $$T^*T=\alpha I-K+F$$ where $K$ is a positive compact operator with $\|K\|\leq \alpha$ and $F$ is a positive finite rank operator such that $KF=FK=0$.
	
		%Suppose dim $N(T)<\infty$, then we are already done.
		%If not, then we have dim $N(T)=\infty$.  Obviously, $T^*T\geq 0$. Then by
		Next, $\sigma(T^*T)=\alpha -\sigma(K-F)$ implies that if there exists an eigenvalue with infinite multiplicity for $T^*T$ then it must be `$\alpha$'. We have the following two cases,
			
		Case(I): $\alpha=0$
		
		 Since $\|K\|\leq \alpha$, we have $K=0$ and $T^*T=F$ is a finite rank operator. Now, $\overline{R(T^*T)}=\overline{R(T^*)}$ implies that dim $R(T^*)<\infty$. From the singular value decomposition theorem \cite[Theorem 4.1, page 248]{goh}, it is easy to observe that an operator is of finite rank if and only if its adjoint is of finite rank. So we must have dim $R(T)<\infty$.
		
		Case(II): $\alpha>0$
		 	
		 In this case, we have dim $N(T)$= dim $N(T^*T) <\infty$. Otherwise, `$0$' has to be the eigenvalue for $T^*T$ with infinite multiplicity and $\alpha=0$, which is not true.		
\end{proof}

\begin{proposition}
		\label{prop:projection}
		Let $P\in \mathcal{B}(H)$ be an orthogonal projection. Then $P\in \mathcal{AM^+}(H)$ iff either $N(P)$ or $R(P)$ is finite dimensional.
\end{proposition}

\begin{proof}
		We have $P\geq 0$. Suppose $R(P)$ is finite dimensional. Then obviously, $P\in \mathcal{AM^+}(H)$. In case $N(P)$ is finite dimensional, then $P=I-P_{N(P)}$ and $P_{N(P)}\geq 0$. Therefore by Theorem \ref{thm:characterization}, $P\in \mathcal{AM^+}(H)$.
\end{proof}

\begin{proposition}
		\label{prop:partialisometry}
		Let $V\in \mathcal{B}(H)$ be partial isometry. Then $V\in \mathcal{AM}(H)$ iff either $N(V)$ or $R(V)$ is finite dimensional.
\end{proposition}

\begin{proof}
		The proof is immediate from Theorem \ref{thm:modulusAM} and Proposition \ref{prop:projection}, if we observe that $|V|=P_{R(V)}$.
\end{proof}

\begin{proposition}
		Let $V\in \mathcal{AM}(H)$ be a partial isometry and $F$ is a finite rank operator. Then $\forall \,\alpha \geq 0$, we have $\alpha V-F \in \mathcal{AM}(H)$.
\end{proposition}
	
\begin{proof}
		Let us denote by $T:=\alpha V-F$. Since $V\in \mathcal{AM}(H)$, we have either $N(V)$ or $R(V)$ is finite dimensional. Firstly, if $R(V)$ is finite dimensional, we are already done because $T$ is a finite rank operator. In the other case, let $N(V)$ be finite dimensional. Then we have $T^*T={\alpha}^2 P_{R(V^*)} - \left[\alpha V^*F+\alpha F^*V+F^*F\right]={\alpha}^2I-\widetilde{F}$ where $\widetilde{F}={\alpha}^2 P_{N(V)}+\alpha V^*F+\alpha F^*V+F^*F$ is a finite rank operator. Now, the proof follows directly from Theorem \ref{thm:finiterank} and Theorem \ref{thm:modulusAM}.
\end{proof}
	
\begin{remark}
		The above Proposition is valid, in particular if we allow $V$ to be an isometry, projection or a co-isometry.
\end{remark}

Using the polar decomposition theorem, Theorem \ref{thm:characterization} can be extended to a more general case, as below.

\begin{theorem}
		Let $T\in \mathcal{B}(H_1, H_2)$. Then the following are equivalent:
		\begin{enumerate}
			\item $T\in \mathcal{AM}(H_1, H_2)$;
			\item There exists a decomposition for $T$ of the form $T:= V\left(\alpha I-K+F\right)$ where  $K\in \mathcal{B}(H_1)$ is a positive compact operator with $\|K\|\leq \alpha$, $F\in \mathcal{B}(H_1)$ is a positive finite rank operator satisfying $KF=FK=0$ and $V\in \mathcal{B}(H_1, H_2)$ is a partial isometry such that $N(V)=N(\alpha I-K+F)$. Moreover, this decomposition is  unique.
		\end{enumerate}
\end{theorem}

\begin{proof}
	
	$(1)\Rightarrow(2)$:
	
	We have $T\in \mathcal{B}(H_1, H_2)$, then by the polar decomposition theorem there exists a unique partial isometry $V\in \mathcal{B}(H_1, H_2)$ such that $T=V|T|$ and $N(V)=N(|T|)$. From Theorem \ref{thm:modulusAM}, $T\in \mathcal{AM}(H_1, H_2)$ implies that $|T|\in \mathcal{AM^+}(H_1)$. Next, by Theorem \ref{thm:characterization}, there exists a decomposition for $|T|$ of the form $T:= \alpha I-K+F$ where  $K\in \mathcal{B}(H_1)$ is a positive compact operator with $\|K\|\leq \alpha$ and $F\in \mathcal{B}(H_1)$ is a positive finite rank operator satisfying $KF=FK=0$. Clearly, we have $N(V)=N(|T|)=N(\alpha I-K+F)$. Next, the uniqueness of the $V$ is clear and the uniqueness of $\alpha, K, F$ comes from the uniqueness of the decomposition for $|T|$.
	
	$(2)\Rightarrow(1)$:
	
	Let $T\in \mathcal{B}(H_1, H_2)$ and has the decomposition of the form given in (2). Let us, denote by $P:=\alpha I-K+F$. Since $\alpha I-K \geq 0$ and $F\geq 0$, we have $P\geq 0$. We also, have $V$ is a partial isometry with $N(V)=N(\alpha I-K+F)=N(P)$.
	Therefore by the uniqueness of the polar decomposition theorem, we must have $|T|=P$. That is, $|T|=\alpha I-K+F$. By applying Theorem \ref{thm:modulusAM} and Theorem \ref{thm:characterization}, it follows that $T\in \mathcal{AM}(H_1, H_2)$.
	%\begin{align*}
	%T^*T & = (\alpha I-K+F)V^*V(\alpha I-K+F)\\
	%     & = (\alpha I-K+F)P_{R(V^*)}(\alpha I-K+F)\\
	%     & = (\alpha I-K+F)    (\alpha I-K+F)
	%\end{align*}	
\end{proof}	

%%%%%%%%%%%%%%%%%%%%%%%%%%%%%%%%%%%%%%%%%%%%%%%%%%%%%%%%%%%%%%%%%%%%%%%%%%%%%%%%%%%%%%%%%
%%%%%%%%%%%%%%%%%%%%%%%%%%%%%%%%%%%%%%%%%%%%%%%%%%%%%%%%%%%%%%%%%%%%%%%%%%%%%%%%%%%%%%%%%
\bibliographystyle{amsplain}

\end{document}